\documentclass{amsart}
\usepackage{amsmath}
\usepackage{amssymb}
\usepackage{amsthm}
\usepackage{amsfonts}
\usepackage{cases}
\usepackage{geometry}
\usepackage{graphicx}
\usepackage{epstopdf}
\usepackage{color}
\usepackage{subfigure}
\usepackage{mathrsfs}
 \newtheorem{theorem}{Theorem}[section]
 \newtheorem{definition}[theorem]{Definition}
 \newtheorem{lemma}[theorem]{Lemma}
 
 \newtheorem{corollary}[theorem]{Corollary}
 
 \newtheorem{remark}{Remark}

\numberwithin{equation}{section}

\def\bR{{\mathbb R}}
\def\dom{\mathrm{dom}}

\begin{document}
 \title[Energy and Laplacian of FIFs]{Energy and Laplacian of Fractal Interpolation Functions}
\author{Xiao-Hui Li}
\address{Department of Mathematics, Zhejiang University, Hangzhou 310027,
China} \email{lxhmath@zju.edu.cn}

\author{Huo-Jun Ruan}
\address{Department of Mathematics, Zhejiang University, Hangzhou 310027,
China} \email{ruanhj@zju.edu.cn}
\thanks{The research of Ruan is supported in part by
the NSFC grant 11271327, and by ZJNSFC grant LR14A010001.\\
\indent Corresponding author: Huo-Jun Ruan}

%    General info
%\subjclass[2010]{Primary 28A80; Secondary 37F99.}

%\date{\bf\today}
%\date{December 1, 2001 and, in revised form, June 22, 2001.}



\begin{abstract}
In this paper, we first characterize the finiteness of fractal interpolation functions (FIFs) on post critical finite self-similar sets. Then we study the Laplacian of FIFs with uniform vertical scaling factors on Sierpinski gasket (SG). As an application, we prove that the solution of the following Dirichlet problem on SG is an FIF with uniform vertical scaling factor $\frac{1}{5}$: $\Delta u=0$ on $SG\setminus \{q_1,q_2,q_3\}$, and $u(q_i)=a_i$, $i=1,2,3$, where $q_i$, $i=1,2,3$, are boundary points of SG.
\end{abstract}

\maketitle

\section{Introduction}

Fractal interpolation functions (FIFs) was introduced by Barnsley \cite{Bar} to model discrete data of natural scenes. Classically, the definition domains of these functions are segments, triangles and rectangles. Recently, \c{C}elik, Ko\c{c}ak and \"{O}zdemir \cite{CKO} and Ruan \cite{Ruan}  defined FIFs on Sierpinski gasket and post critically finite (p.c.f.) self-similar sets. Furthermore, Ruan \cite{Ruan} and Ri and Ruan \cite{RR} studied analytic properties of these functions, including energy, normal derivative and Laplacian. These results imply that the class of FIFs provides a large collection of explicit functions with finite energy and they are suitable to the theory of analysis on fractals.  In this paper, we will continue these work.

First, we recall some basic definitions of p.c.f. self-similar sets and FIFs. Let $P_i$, $i=1,2,\ldots,N$, be contractive similitudes in $\bR^n$. Then there exists a unique nonempty compact subset $K$ of $\bR^n$ satisfying $K=\bigcup_{i=1}^N P_i(K)$. Define $\Sigma_N=\{1,2,\ldots,N\}$ and $\Sigma_N^m=\{\omega_1 \omega_2\cdots \omega_m|\, \omega_j\in \Sigma_N \mbox{ for any }j\}$ for any $m\geq 1$. For $\omega \in \Sigma_N^m$, we say that $\omega$ is a word with length $|\omega|:=m$.  Define $\Sigma_N^*=\bigcup_{m=1}^\infty \Sigma_N^m$. For convention, we denote $\Sigma_N^0=\{\vartheta\}$, where $\vartheta$ is the empty word with length $0$.

Let $q_i$, $1\leq i\leq N$, be the fixed point of $P_i$. For $\omega=\omega_1\cdots\omega_m \in \Sigma_N^*$, we define
\begin{equation*}
  P_\omega=P_{\omega_1}\circ\cdots\circ P_{\omega_m}, \quad q_\omega=P_{\omega_1\cdots\omega_{m-1}}(q_{\omega_m}).
\end{equation*}
Then we can present the following definition of  p.c.f. self-similar sets  by Strichartz \cite{Str}, which is weaker than the original version by Kigami \cite{Kig01}.
\begin{definition}\label{def: pcf-def}
  The self-similar set $K$ defined above is call \emph{post critically finite}, or \emph{p.c.f.} for short, if $K$ is connected and there exists a finite set $V_0\subset K$ called the boundary, such that
  \begin{equation}\label{eq:pcf-cond}
    P_\omega(K) \cap P_{\omega^\prime} (K) = P_\omega(V_0) \cap P_{\omega^\prime} (V_0), \quad \mbox{for } \omega\not=\omega^\prime \mbox{ with } |\omega|=|\omega^\prime|,
  \end{equation}
  with the intersection disjoint from $V_0$. Moreover, we require that $V_0$ is the minimum subset of $\{q_1,q_2,\ldots,q_N\}$ which satisfies \eqref{eq:pcf-cond}. Without loss of generality, we suppose $V_0=\{q_1,q_2,\ldots,q_{N_0}\}$ for $N_0\leq N$.
\end{definition}

Let $V_m=\bigcup_{|\omega|=m} P_\omega(V_0)$ for any positive integer $m$ and $V_*=\bigcup_{m=1}^\infty V_m$.
We can define graph structure on $V_*$. Define $\Gamma_0$ to be the complete graph on vertex set $V_0$. For $m \geq 1$,  we define the graph $\Gamma_m$ on $V_m$ as follows: for any $x,y\in V_m$, the edge relation $x\sim_m y$ to hold if and only if there exists $\omega \in \Sigma_N^m$ such that $x,y \in P_\omega(V_0)$.

Let $\{c_{ij}\}_{1\leq i<j\leq N_0}$ and $\{r_i\}_{1\leq i\leq N}$ be all positive real numbers. Denote $r_\omega=r_{\omega_1}r_{\omega_1}\cdots r_{\omega_m}$ for any $\omega\in \Sigma_N^m$.    For $m=0,1,\ldots,$ we define graph energy $\mathcal{E}_{m}$ of a function $u$ on $V_m$ by
\begin{equation}\label{eq:Em-def}
  \mathcal{E}_{m}(u) = \sum_{x\sim_m y} c_m(x,y) (u(x)-u(y))^2,
\end{equation}
where $c_m(x,y)=r_\omega^{-1}c_{ij}$ if $x=P_\omega q_i, y=P_\omega q_j$ with $i<j$. If the graph energy sequences $\{\mathcal{E}_m\}$ satisfies
\begin{equation}\label{eq:energy-condition}
  \mathcal{E}_{m-1}(u) = \min \mathcal{E}_m (\widetilde{u}),
\end{equation}
where the minimum is taken over all $\widetilde{u}$ satisfying $\widetilde{u}|_{V_{m-1}}=u$ for all $u:\, K\to \bR$ and for all $m\geq 1$, then we call
\begin{equation}\label{eq:energy-def}
  \mathcal{E}(u) = \lim_{m\to\infty} \mathcal{E}_m(u)
\end{equation}
the \emph{energy} of $u$ on $K$ w.r.t. $\{c_{ij}\}$ and $\{r_i\}$, or the energy of $u$ if no confusion will occur. We call $u$ a \emph{harmonic function} if $\mathcal{E}_{m-1}(u)=\mathcal{E}_m(u)$ for all $m\geq 1$. By \eqref{eq:energy-condition}, for each function $u$ on $V_*$, the sequence $\{\mathcal{E}_m(u)\}_{m=0}^\infty$ is increasing. We call $u$ has \emph{finite energy} if $\lim_{m\to \infty} \mathcal{E}_m(u) < +\infty$.

It is well known that for Sierpinski gasket (SG), we can take $c_{ij}=1$ for all $1\leq i<j\leq 3$ and $r_i=3/5$ for all $i$ so that the corresponding graph energy sequence satisfies \eqref{eq:energy-condition}. Throughout the paper, we always suppose that the energy $\mathcal{E}$ on $K$ is well defined.

From \cite{Ruan}, we have the following result.
\begin{theorem}\label{th:FIF-def-pcf}
  Let $K$ be the p.c.f. self-similar set determined by $\{P_i\}_{i\in \Sigma}$.
  Let $B:V_1 \rightarrow \mathbb{R}$ be a given function. For any given numbers $d_i \in (-1,1)$, $i\in \Sigma_N$, there exists a unique continuous function $f:K \rightarrow \mathbb{R}$, such that $f|_{V_1}=B$ and
  \begin{equation}\label{eq:uni-FIF}
   f(P_i(x))=d_i f(x)+h_i(x)
  \end{equation}
  for $x \in K$, where $h_i$ are harmonic functions on $K$ for all $i\in\Sigma_N$. $f$ is called a \emph{fractal interpolation function} defined by \emph{basic function} $B$ and \emph{vertical scaling factors} $d_i$, $i\in \Sigma_N$.
\end{theorem}

In \cite{Ruan}, Ruan proved that the FIF $f$ defined above has finite energy if $\sum_{i=1}^N r_i^{-1}d_i^2<1/2$. In \cite{RR}, Ri and Ruan focused the  FIFs on SG with following conditions: $B|_{V_0}=0$, $B|_{V_1\setminus V_0}=1$ and $d_1=d_2=d_3$. These FIFs are called \emph{uniform FIFs} on SG. They showed that the uniform FIF has finite energy if and only if $d^2<1/5$, where $d$ is the common value of $d_i$. Another main result in \cite{RR} is:
\begin{theorem}\label{thm: Lap-RR}
  Let $f$ be the uniform FIF on SG with vertical scaling factor $d=\frac{1}{5}$. Let $\alpha$ be a given real number. Then $\frac{-\alpha f}{15}$ is the unique solution of the following Dirichlet problem: $u|_{V_0}=0$, and $\Delta u(x)=\alpha$ for all $x\in SG\setminus V_0$.
\end{theorem}

In this paper, we first prove that an FIF $f$ on p.c.f. self-similar sets has finite energy if and only if $\sum_{i=1}^N r_i^{-1} d_i^2<1$. This completely characterize the finiteness of energy of FIFs. Then we generalize Theorem~\ref{thm: Lap-RR} to FIFs on SG with the condition $d_1=d_2=d_3$, while there is no restriction on the basic function. We call such function \emph{an FIF on SG with uniform vertical scaling factor $d$}, where $d$ is the common value of $d_i$.

The paper is organized as follows.
%In Section 2, we recall some basic notations in analysis on p.c.f. self-similar sets.
In Section 2, we present the sufficient and necessary condition such that FIFs on p.c.f. self-similar sets have finite energy. In Section 3, we study the existence of Laplacian of FIFs on SG with uniform vertical scaling factor, and generalize  Theorem~\ref{thm: Lap-RR}.

\section{Characterization of finite energy of FIFs on p.c.f. self-similar sets}
Let $\mathcal{E}$ be the energy on a p.c.f. self-similar set $K$ defined by \eqref{eq:Em-def} and \eqref{eq:energy-def}.

Given two function $u,v$ on $V_m$, where $m\geq 0$, we define
\begin{equation*}
  \mathcal{E}_m(u,v) = \sum_{x\sim_m y} c_m(x,y) (u(x)-u(y))(v(x)-v(y)).
\end{equation*}

We need the following basic lemma to characterize the finiteness of energy. Essentially, this is equivalent to Lemma~3.2.16 in \cite{Kig01}. We sketch the proof for the completeness.
\begin{lemma}\label{3lem:energy}
Let $u,v$ be defined on $V_m$, let $\widetilde{u} $ be the harmonic extension of $u$, and let $v^\prime$ be any extension of $v$ to $V_{m+1}.$ Then
\begin{equation*}
 \mathcal{E}_{m+1}(\widetilde{u},v^\prime)=\mathcal{E}_m(u,v).
\end{equation*}
\end{lemma}
\begin{proof}
  Similarly as the case of SG (see the proof of Lemma~1.3.1 in \cite{Str}), it suffices to show that $\mathcal{E}_{m+1}(\widetilde{u},v^{\prime\prime})=0$ for all functions $v^{\prime\prime}$ on $V_{m+1}$ with $v^{\prime\prime}=0$ on $V_m$. However, this has already been proved in the proof of Lemma~3.2.16 in \cite{Kig01}. Thus the theorem holds.
\end{proof}

\begin{remark}
   Define an $N\times N$ matrix $C$ by $C_{ji}=C_{ij}=c_{ij}$ for $1\leq i<j\leq N$ and $C_{ii}=-\sum_{j\not=i} C_{ij}$. Denote $\mathbf{r}=(r_1,r_2,\ldots,r_N)$. Then the existence of energy is equivalent to that $(C,\mathbf{r})$ is a harmonic structure on $(K,\Sigma_N,\{P_i\}_{i\in \Sigma_N})$. Thus the results in \cite{Kig01} are applicable. See \cite{Kig01} for details.
\end{remark}

The following theorem characterize the  finiteness of energy of FIFs.

\begin{theorem}
  Let $f$ be an FIF on the p.c.f. self-similar set $K$ determined as in Theorem~\ref{th:FIF-def-pcf}. Then  $\mathcal{E}(f)<\infty$ if and only if either $\mathcal{E}_1(f)=\mathcal{E}_0(f)$ or $\sum_{k=1}^N r_k^{-1} d_k^2<1$. Furthermore, in case that $\mathcal{E}_1(f)=\mathcal{E}_0(f)$,  $f$ is a harmonic function so that $\mathcal{E}(f)=\mathcal{E}_0(f)$; and in case that $\sum_{k=1}^N r_k^{-1}d_k^2<1,$
 \begin{equation*}
 \mathcal{E}(f)=\mathcal{E}_0(f) + \frac{1}{1-\Big(\sum_{k=1}^N r_k^{-1} d_k^2\Big)} (\mathcal{E}_1(f)-\mathcal{E}_0(f)).
 \end{equation*}
\end{theorem}
\begin{proof}
By definition, for $m \geq 1,$
\begin{align*}
\mathcal{E}_m(f)&=\sum_{\omega \in \Sigma_N^m}\sum_{1\leq i<j\leq N_0} r_\omega^{-1}c_{ij}\left(f(q_{\omega i})-f(q_{\omega j})\right)^2\\
&=\sum_{\tau \in \Sigma_N^{m-1}}\sum_{1\leq i<j\leq N_0} r_{\tau}^{-1}\Big(\sum_{k\in\Sigma_N}r_k^{-1}c_{ij}\left(f(q_{k\tau i})-f(q_{k\tau j})\right)^2\Big),
\end{align*}
where we define $q_{k\vartheta i} = q_{ki}$ in case that $m=1$.
For any $\tau \in \Sigma_N^{m-1}$ and $1\leq i<j\leq N_0$, we have
\begin{align*}
 \sum_{k\in \Sigma_N}\left(f(q_{k\tau i})-f(q_{k\tau j})\right)^2 = & \sum_{k=1}^N\Big(d_k(f(q_{\tau i})-f(q_{\tau j}))+h_k(q_{\tau i})-h_k(q_{\tau j})\Big)^2\\
   = &\sum_{k=1}^N \Big( d_k^2(f(q_{\tau i})-f(q_{\tau j}))^2+ (h_k(q_{\tau i})-h_k(q_{\tau j}))^2 \\
   & +2d_k(f(q_{\tau i})-f(q_{\tau j}))(h_k(q_{\tau i})-h_k(q_{\tau j})) \Big).
\end{align*}
It follows that
\begin{align*}
 \mathcal{E}_m(f)
   =&  \sum_{k=1}^N r_k^{-1} \Big( d_k^2 \mathcal{E}_{m-1}(f)+ \mathcal{E}_{m-1}(h_k)+2d_k \mathcal{E}_{m-1}(f,h_k)\Big)\\
   =&   \sum_{k=1}^N r_k^{-1} \Big( d_k^2 \mathcal{E}_{m-1}(f)+ \mathcal{E}_{0}(h_k)+2d_k \mathcal{E}_{0}(f,h_k)\Big)
\end{align*}
where the last equality follows from Lemma~\ref{3lem:energy}. Noticing that from the above equalities, we have
$\mathcal{E}_1(f)= \sum_{k=1}^N r_k^{-1} \Big( d_k^2 \mathcal{E}_{0}(f)+ \mathcal{E}_{0}(h_k)+2d_k \mathcal{E}_{0}(f,h_k)\Big)$.
Thus for $m \geq 1$, we have
 \begin{equation}\label{eq:Em-iter}
 \mathcal{E}_m(f)=\Big(\sum_{k=1}^N r_k^{-1}d_k^2\Big)\Big(\mathcal{E}_{m-1}(f)-\mathcal{E}_0(f)\Big) + \mathcal{E}_1(f).
 \end{equation}

 Denote $\delta=\sum_{k=1}^N r_k^{-1}d_k^2$. In case that $\delta=1$, we have $\mathcal{E}_m(f)=\mathcal{E}_{0}(f)+m\Big(\mathcal{E}_1(f)-\mathcal{E}_0(f)\Big)$. Thus, $\mathcal{E}(f)<\infty$ if and only if $\mathcal{E}_1(f)=\mathcal{E}_0(f)$. And under this condition, we have $\mathcal{E}_m(f)=\mathcal{E}_0(f)$ for all $m$ so that $f$ is a harmonic function.

 In case that $\delta\not=1$. From \eqref{eq:Em-iter}, we can obtain that for all $m\geq 1$,
 \begin{equation*}
   \mathcal{E}_m(f) + \frac{\delta}{1-\delta}\mathcal{E}_0(f) -\frac{1}{1-\delta} \mathcal{E}_1(f) = \delta\Big( \mathcal{E}_{m-1}(f) + \frac{\delta}{1-\delta}\mathcal{E}_0(f) -\frac{1}{1-\delta} \mathcal{E}_1(f)\Big)
 \end{equation*}
so that
 \begin{equation}\label{eq:Em-formula}
   \mathcal{E}_m(f)=\frac{\delta^{m}}{1-\delta} \Big(\mathcal{E}_0(f)-\mathcal{E}_1(f)\Big) + \mathcal{E}_0(f) + \frac{1}{1-\delta} (\mathcal{E}_1(f)-\mathcal{E}_0(f)).
 \end{equation}
 Thus, $f$ has finite energy if and only if either $\mathcal{E}_1(f)=\mathcal{E}_0(f)$ or $\delta<1$. If $\mathcal{E}_1(f)=\mathcal{E}_0(f)$, we can see from \eqref{eq:Em-formula} that $\mathcal{E}_m(f)=\mathcal{E}_0(f)$ for all $m$ so that $f$ is a harmonic function. If $\delta<1$, using \eqref{eq:Em-formula} again, we can see that $\mathcal{E}(f)=\mathcal{E}_0(f) + \frac{1}{1-\delta} (\mathcal{E}_1(f)-\mathcal{E}_0(f))$. By the definition of $\delta$, we know that the theorem holds.
\end{proof}

\begin{corollary}\label{corol:energy-SG}
Let $f$ be an FIF on SG with uniform vertical scaling factor $d$. Then $\mathcal{E}(f)<\infty$ if and only if either $\mathcal{E}_1(f)=\mathcal{E}_0(f)$ or $\sum_{k=1}^3 d_k^2<\frac{3}{5}$. Furthermore, in case that $\mathcal{E}_1(f)=\mathcal{E}_0(f)$,  $f$ is a harmonic function so that $\mathcal{E}(f)=\mathcal{E}_0(f)$; and in case that $\sum_{k=1}^3 d_k^2<\frac{3}{5},$
 \begin{equation*}
 \mathcal{E}(f)=\mathcal{E}_0(f) + \frac{1}{1-\frac{5}{3}\Big(\sum_{k=1}^3 d_k^2\Big)} (\mathcal{E}_1(f)-\mathcal{E}_0(f)).
 \end{equation*}
\end{corollary}

We remark that $\mathcal{E}_1(f)=\mathcal{E}_0(f)$ on SG if and only if $f$ satisfies the ``$\frac{1}{5}-\frac{2}{5}$ rule" at the points in $V_1\setminus V_0$. For details, please see Section~1.3 in \cite{Str}.

\section{Laplacian of FIFs on Sierpinski gasket }
In this section, we will discuss the Laplacian of FIFs with uniform vertical scaling factor $d$ on Sierpinski gasket. In this case, \eqref{eq:uni-FIF} can be replaced by
\begin{equation}\label{eq:FIF-def}
  f(P_i(x))=d\cdot f(x)+h_i(x), \quad x\in SG, \; i=1,2,3,
\end{equation}
where $h_i$, $i=1,2,3,$ are harmonic functions. The following property is the well-known ``$\frac{1}{5}-\frac{2}{5}$ rule" of harmonic functions on SG.
\begin{theorem}[\cite{Kig01,Str}]
Let $h$ be a harmonic function on SG. Let $(i,j,k)$ be a permutation of $(1,2,3)$. Then, for any $\omega\in\Sigma_3^*\cup \{\emptyset\}$, we have
\begin{equation*}
  h(q_{\omega ij})=\frac{2}{5}h(q_{\omega i})+\frac{2}{5}h(q_{\omega j})+\frac{1}{5}h(q_{\omega k}).
\end{equation*}
\end{theorem}

Now we recall some definitions of Laplacian on SG. For any continuous function $u$ on SG, we define the \emph{graph Laplacian} $\Delta_m$ for positive integers $m$ by
\begin{equation*}
  \Delta_mu(x)=\sum_{y\sim_mx}(u(y)-u(x)), \quad x\in V_m\backslash V_0.
\end{equation*}

\begin{definition}
Suppose $g$ is a continuous function on SG. We say $u \in \dom \Delta$ with $\Delta u=g$ if
\begin{equation*}
 \frac{3}{2}5^{m}\Delta_{m}u(x)
\end{equation*}
converges uniformly to $g$ on  $V_*\backslash V_0$ as $m$ goes to infinity.
\end{definition}

It is well known that $u$ has finite energy if $u\in\dom \Delta$.
For harmonic function $h$, we have $\Delta_m h(x)=0$ for any $m\in \mathbb{Z}^+$ and any $x\in V_m\backslash V_0$  so that $\Delta h=0$. Please see \cite{Kig01,Str} for details.

 \begin{lemma}\label{lem:harm-recurrent}
   Let $h$ be a harmonic function on SG and $(i,j,k)$ be a permutation of $(1,2,3)$.  Then for any positive integer $m$,
   \begin{equation}\label{5eq:tempsum}
   h(q_{i^{m}j})+h(q_{i^{m}k})=2h(q_i)+\left(\frac{3}{5}\right)^m\left(h(q_j)+h(q_k)-2h(q_i)\right),
  \end{equation}
  where we denote $i^m$ to be the word $\underbrace{i\cdots i}_m$.
   \end{lemma}
   \begin{proof}
   From ``$\frac{1}{5}-\frac{2}{5}$ rule", we have
   \begin{equation*}
\begin{cases}
    h(q_{i^{m}j}) =\frac{2}{5}h(q_i)+\frac{2}{5}h(q_{i^{m-1}j})+\frac{1}{5}h(q_{i^{m-1}k}) \\
     h(q_{i^{m}k}) =\frac{2}{5}h(q_i)+\frac{1}{5}h(q_{i^{m-1}j})+\frac{2}{5}h(q_{i^{m-1}k})
\end{cases}
\end{equation*}
so that
\begin{align*}
   h(q_{i^{m}j})+h(q_{i^{m}k})&=\frac{4}{5}h(q_i)+\frac{3}{5}\left(h(q_{i^{m-1}j})+h(q_{i^{m-1}k})\right) \nonumber \\
   &=\frac{4}{5}h(q_i)\left(1+\frac{3}{5}+\cdots+\left(\frac{3}{5}\right)^{m-1}\right)+\left(\frac{3}{5}\right)^m(h(q_j)+h(q_k)) \nonumber \\
   &=2h(q_i)+\left(\frac{3}{5}\right)^m\left(h(q_j)+h(q_k)-2h(q_i)\right).
  \end{align*}
   \end{proof}

    In the sequel of the paper, for any function $u$ on $V_1$, we denote
 \begin{align*}
 &\Delta_0 u(q_i)=\sum_{y\sim_0q_i}(u(y)-u(q_i)),\\
  &\Delta_1^i u(x)=\sum_{ y\sim_1x \ \text{and} \  y\in P_i(V_0)}(u(y)-u(x)),\quad x\in V_1, \;  i=1,2,3.
 \end{align*}

 \begin{lemma}\label{lem:f+f}
  Let $f$ be an FIF with uniform vertical scaling factor $d\not=\frac{3}{5}$ on SG. Assume that $(i,j,k)$ be a permutation of $(1,2,3)$.  Then for  any nonnegative integer $m$,
 \begin{equation}\label{5eq:sumf23}
   f(q_{i^{m}j})+f(q_{i^{m}k})=2f(q_i)+d^m\Delta_0f(q_i)+\frac{\Delta_1^if(q_{i})-d\Delta_0f(q_i)}{3/5-d}\left[(\frac{3}{5})^m-d^m\right].
%   +\left(-\frac{\Delta_1^if(q_{i})-d\Delta_0f(q_i)}{3/5-d}\right)d^m.
 \end{equation}
\end{lemma}
 \begin{proof}
 Let $h_i$, $i=1,2,3$, be harmonic functions defined by \eqref{eq:FIF-def}. Then
  \begin{equation}\label{eq:hi-def}
  h_i(q_j)=f(q_{ij})-d\cdot f(q_j), \quad i,j=1,2,3.
  \end{equation}
  By Lemma~\ref{lem:harm-recurrent}, for all $m\geq 0$,
 \begin{equation}\label{eq:hiij}
    h_i(q_{i^{m}j})+h_i(q_{i^{m}k})%=2h_i(q_i)+\left(h_i(q_j)-2h_i(q_i)+h_i(q_k)\right)\left(\frac{3}{5}\right)^m\\
    =2f(q_{i})-2df(q_i)+\left(\frac{3}{5}\right)^m\left(\Delta_1^if(q_{i})-d\Delta_0f(q_i)\right).
 \end{equation}

 Denote $\delta_i=\Delta_1^if(q_{i})-d\Delta_0f(q_i), \ i=1,2,3$. From the above equality and \eqref{eq:FIF-def},
 \begin{align*}
   f(q_{i^{m}j})+f(q_{i^{m}k}) &= d  \left( f(q_{i^{m-1}j})+f(q_{i^{m-1}k})\right)+h_i(q_{i^{m-1}j})+h_i(q_{i^{m-1}k}) \nonumber\\
    &=d \left( f(q_{i^{m-1}j})+f(q_{i^{m-1}k})\right)+2f(q_{i})-2df(q_i)+\delta_i\left(\frac{3}{5}\right)^{m-1}.
 \end{align*}
It follows that
 \begin{align*}
   f(q_{i^{m}j})&+f(q_{i^{m}k}) -2f(q_i)-\frac{\delta_i}{3/5-d}\left(\frac{3}{5}\right)^{m}  \\ &= d  \left(f(q_{i^{m-1}j})+f(q_{i^{m-1}k})-2f(q_i)-\frac{\delta_i}{3/5-d}\left(\frac{3}{5}\right)^{m-1} \right)\\
    &=d^m  \left( f(q_j)+f(q_k)-2f(q_i)-\frac{\delta_i}{3/5-d}\right).
 \end{align*}
 Thus the lemma holds.
 \end{proof}

\begin{lemma}\label{lem:delfm+1}
 Let $f$ be an FIF with uniform vertical scaling factor $d\not=\frac{3}{5}$ on SG.  Then for  any nonnegative integer $m$,
 \begin{equation}\label{eq:delm+1}
\Delta_{m+1}f(q_{12})=\frac{d(\Delta_1^2 f(q_2)+\Delta_1^1 f(q_1)+\frac{3}{5}\Delta_0 f(q_3))}{3/5-d}\cdot\left[ (\frac{3}{5})^m - d^m \right] + \Delta_1 f(q_{12}) \cdot (\frac{3}{5})^m.
\end{equation}
\end{lemma}
\begin{proof}
 Let $h_i$, $i=1,2,3$ be harmonic functions defined by \eqref{eq:FIF-def}. Then for all $m\geq 0$,
\begin{align*}
\Delta_{m+1}f(q_{12})=&f(q_{12^{m}1})+f(q_{12^{m}3})+f(q_{21^{m}2})+f(q_{21^{m}3})-4f(q_{12})\nonumber\\
=&d\left\{f(q_{2^{m}1})+f(q_{2^{m}3})+f(q_{1^{m}2})+f(q_{1^{m}3})\right\}\nonumber\\
 &+h_1(q_{2^{m}1})+h_1(q_{2^{m}3})+h_2(q_{1^{m}2})+h_2(q_{1^{m}3})-4f(q_{12}).
\end{align*}
Noticing that $q_{12}=q_{21}$,  $\Delta_0f(q_1)+\Delta_0f(q_2)+\Delta_0f(q_3)=0$ and  $\Delta_1^1f(q_{12})+\Delta_1^2f(q_{21})=\Delta_1f(q_{12})$.
Thus, from Lemma~\ref{lem:f+f}, we have
\begin{align*}
  &d\left\{f(q_{2^{m}1})+f(q_{2^{m}3})+f(q_{1^{m}2})+f(q_{1^{m}3})-2(f(q_1)+f(q_2))\right\}\\
    &\hspace{2em}= -d^{m+1}\Delta_0 f(q_3) + \frac{d (\Delta_1^2 f(q_2)+\Delta_1^1 f(q_1)
      + d\Delta_0 f(q_3) )}{3/5-d}  \cdot \left[ (\frac{3}{5})^m-d^m\right].
\end{align*}
From \eqref{eq:hi-def} and  Lemma~\ref{lem:harm-recurrent}, if $(i,j,k)$ is a permutation of $(1,2,3)$, then
 \begin{equation}\label{eq:hiji}
         h_i(q_{j^{m}i})+h_i(q_{j^{m}k})%=2h_i(q_j)+\left(h_i(q_i)-2h_i(q_j)+h_i(q_k)\right)\left(\frac{3}{5}\right)^m \nonumber\\
                                       =2f(q_{ij})-2df(q_j)+\left(\Delta_1^if(q_{ij})-d\Delta_0f(q_j)\right)\left(\frac{3}{5}\right)^m.
 \end{equation}
It follows that
\begin{align*}
  &h_1(q_{2^{m}1})+h_1(q_{2^{m}3})+h_2(q_{1^{m}2})+h_2(q_{1^{m}3})-4f(q_{12})+ 2d(f(q_2)+f(q_1))\\
  &\hspace{2em}=(\Delta_1 f(q_{12}) + d \Delta_0 f(q_3) ) \cdot (\frac{3}{5})^m.
\end{align*}
From $1+\frac{d}{3/5-d}=\frac{3/5}{3/5-d}$, we know that the corollary holds.
%we obtain that
%\begin{align*}
%  \Delta_{m+1} f(q_{12}) = \frac{d(\Delta_1^2 f(q_2)+\Delta_1^1 %f(q_1)+\frac{3}{5}\Delta_0 f(q_3))}{3/5-d}\cdot\left[ (\frac{3}{5})^m - d^m %\right] + \Delta_1 f(q_{12}) \cdot (\frac{3}{5})^m.
%\end{align*}
\end{proof}

   \begin{lemma}\label{5lem:DF}
   Let $f$ be an FIF with uniform vertical scaling factor $d$ on SG. Then, for any $\omega \in \Sigma_3^*,$ any nonnegative integer $m$ and distinct $i,j \in \{1,2,3\},$
   \begin{equation*}
     \Delta_{|\omega|+m+1}f(q_{\omega i j})=d^{|\omega|}\Delta_{m+1}f(q_{ij}).
   \end{equation*}
   \end{lemma}
   \begin{proof}
      We can use the same proof as Lemma 6.2 in \cite{RR}. Thus we omit the details.
   \end{proof}

   \begin{theorem}
   Let $f$ be an FIF with uniform vertical scaling factor $d$ on SG. If $\Delta f(p)< \infty$ for all $p \in V_*\backslash V_0 $,   then $f$ satisfies one of the following conditions:
   \begin{enumerate}
     \item   $f$ satisfies the ``$\frac{1}{5}-\frac{2}{5}$ rule" at the points in $V_1\setminus V_0$;
     \item   $d=\frac{1}{5}$ and $f(q_{12})+\frac{1}{5}f(q_3)=f(q_{13})+\frac{1}{5}f(q_2)=f(q_{23})+\frac{1}{5}f(q_1).$
   \end{enumerate}
   Furthermore, in case 1, $f$ is a harmonic function so that $\Delta f=0$, and in case 2,
   \begin{equation}\label{eq:Delta-f-formula}
     \Delta f=3\big(2f(q_1)+2f(q_2)+f(q_3)-5f(q_{12})\big).
   \end{equation}
  \end{theorem}

   \begin{proof}
     From Corollary~\ref{corol:energy-SG}, $\mathcal{E}(f)<\infty$ if and only if either $\mathcal{E}_1(f)=\mathcal{E}_0(f)$ or $|d|<\frac{1}{\sqrt{5}}$. Furthermore, in case that $\mathcal{E}_1(f)=\mathcal{E}_0(f)$, $f$ is a harmonic function so that $\Delta f=0$ on $SG\backslash V_0$. Notice that $\mathcal{E}_1(f)=\mathcal{E}_0(f)$ if and only if $f$ satisfies the ``$\frac{1}{5}-\frac{2}{5}$ rule" at the points in $V_1\setminus V_0$. Thus, in the sequel of the proof, we assume  that $|d|<\frac{1}{\sqrt{5}}$.

   For any $p \in V_*\backslash V_1 $, there exists $\omega=\omega_1\cdots \omega_n \in \Sigma^*$ and distinct $i,j \in \{1,2,3\}$ such that $p=q_{\omega ij}.$ By Lemma \ref{5lem:DF}, we have
   \begin{eqnarray}
     \Delta f(p) &=& \Delta f(q_{\omega ij})= \frac{3}{2}\lim_{m\rightarrow\infty}5^{m+n+1}\Delta_{m+n+1}f(q_{\omega ij})\notag\\
      &=& \frac{3}{2}\lim_{m\rightarrow\infty}5^{m+n+1}d^n\Delta_{m+1}f(q_{ ij})=(5d)^n\Delta f(q_{ ij}). \label{eq:Delta-f-p}
   \end{eqnarray}
   Thus we only need to find the condition such that $\Delta f(q_{12}),\Delta f(q_{13}),\Delta f(q_{23})<\infty.$

From Lemma~\ref{lem:delfm+1} and using the symmetry, we can see that for distinct $i,j\in\{1,2,3\}$,
\begin{align}
\Delta f(q_{ij})&=\frac{3}{2}\lim_{m\rightarrow\infty}5^{m+1}\Delta_{m+1}f(q_{ij}) \notag\\
&=\frac{15}{2}\lim_{m\rightarrow\infty}\left\{ \frac{d(\Delta_1^i f(q_i)+\Delta_1^j f(q_j)+\frac{3}{5}\Delta_0 f(q_{6-i-j}))}{3/5-d}\cdot\left[ 3^m - (5d)^m \right] + \Delta_1 f(q_{ij}) \cdot 3^m \right\}.\label{eq:Delta-f-qij}
\end{align}

   For simplicity, we denote $f(q_i)$ by $x_i$ for $i=1,2,3$. For distinct $i,j\in \{1,2,3\}$, we denote $f(q_{ij})$ by $y_{6-i-j}$ and
   \begin{align}\label{def:alpha-ij-def}
     \alpha_{ij}=\Delta_1^i f(q_i)+\Delta_1^j f(q_{j})+\frac{3}{5}\Delta_0 f(q_{6-i-j})
     =-\frac{7}{5}x_i-\frac{7}{5}x_j-\frac{6}{5}x_{6-i-j}+y_i+y_j+2y_{6-i-j}.
   \end{align}
   Then $\Delta f(q_{ij})$ exist for all distinct $i,j$ if and only if
   \begin{equation}\label{eq:final-eq}
     \frac{d}{3/5-d}\alpha_{ij}+\Delta_1 f(q_{ij})=0
   \end{equation}
   for all distinct $i,j$,
   and one of the following condition holds: $|d|<1/5$, $d=1/5$ or $\alpha_{ij}=0$ for all distinct $i,j$.

   Notice that $\alpha_{ij}=0$ for all distinct $i,j$ implies that
   \begin{equation*}
      A_1 (y_1, y_2, y_3)^T = B_1 (x_1, x_2, x_3)^T, \quad \mbox{where}
   \end{equation*}
%   where
   \begin{equation*}
     A_1=\left(\begin{array}{ccc}  1 & 1 & 2 \\ 1 & 2 & 1 \\ 2 & 1 &1  \end{array}\right), \quad
     B_1=\left(\begin{array}{ccc}  \frac{7}{5} & \frac{7}{5} & \frac{6}{5} \\ \frac{7}{5} & \frac{6}{5} & \frac{7}{5} \\ \frac{6}{5} & \frac{7}{5} & \frac{7}{5}  \end{array}\right).
   \end{equation*}
   Meanwhile, $\Delta_1 f(q_{ij})=0$ for all distinct $i,j$ implies that
   \begin{equation*}
      A_2 (y_1, y_2, y_3)^T = B_2 (x_1, x_2, x_3)^T, \quad \mbox{where}
   \end{equation*}
%   where
   \begin{equation*}
     A_2=\left(\begin{array}{ccc}  1 & 1 & -4 \\ 1 & -4 & 1 \\ -4 & 1 & 1  \end{array}\right), \quad
     B_2=\left(\begin{array}{ccc}  -1 & -1 & 0 \\ -1 & 0 & -1 \\ 0 & -1 & -1  \end{array}\right).
   \end{equation*}

   By direct computation, it is easy to  see that
   \begin{equation*}
     A_1^{-1} B_1 = A_2^{-1} B_2 =  \left(\begin{array}{ccc}  \frac{1}{5} & \frac{2}{5} & \frac{2}{5} \\ \frac{2}{5} & \frac{1}{5} & \frac{2}{5} \\ \frac{2}{5} & \frac{2}{5} & \frac{1}{5}  \end{array}\right).
   \end{equation*}

   Notice that $\det(\lambda A_1+A_2)=-2(\lambda-5)^2(2\lambda-1)$. Thus $\det(\frac{d}{3/5-d} A_1+A_2)=0$ if and only if $d=\frac{1}{5}$ or $\frac{1}{2}$. From $|d|<\frac{1}{\sqrt{5}}$, we know that the unique reasonable solution for  $\det(\frac{d}{3/5-d} A_1+A_2)=0$ is $d=\frac{1}{5}$.
   Hence, in the case that $|d|<\frac{1}{\sqrt{5}}$ and $d\not=\frac{1}{5}$,
   $\Delta f(q_{ij})$ exist for all distinct $i,j$ if and only if $(y_1,y_2,y_3)^T=A_1^{-1}B_1 (x_1,x_2,x_3)^T$, i.e.,
   $f$ satisfies the ``$\frac{1}{5}-\frac{2}{5}$ rule" at the points in $V_1\setminus V_0$. Thus, this is a subcase of $\mathcal{E}_1(f)=\mathcal{E}_0(f)$.

   In the case that $d=\frac{1}{5}$,   we know from \eqref{eq:final-eq} that  $\Delta f(q_{ij})$ exist for all distinct $i,j$ if and only if
   $$ \left(\begin{array}{ccc} 1 & 1 & -2 \\ 1 & -2 & 1 \\ -2& 1& 1\end{array}\right)\Big((y_1,y_2,y_3)^T+\frac{1}{5}(x_1,x_2,x_3)^T\Big)=0,$$
   which is equivalent to
  \begin{equation}\label{eq:y123-equality}
   y_1+\frac{1}{5}x_1=y_2+\frac{1}{5}x_2=y_3+\frac{1}{5}x_3.
 \end{equation}
 Furthermore, in this case, we can obtain from \eqref{def:alpha-ij-def} that
 \begin{align*}
   \alpha_{ij}&=(y_i+\frac{1}{5}x_i) + (y_j+\frac{1}{5}x_j) + 2(y_{6-i-j}+\frac{1}{5}x_{6-i-j}) - \frac{8}{5}(x_i+x_j+x_{6-i-j}) \\
   &=4(y_1+\frac{1}{5}x_1) - \frac{8}{5}(x_1+x_2+x_3)=4\left( y_1 - \frac{1}{5}(2x_1+2x_2+x_3)\right).
 \end{align*}
 Combining this with \eqref{eq:Delta-f-p} and \eqref{eq:Delta-f-qij}, we have
 $\Delta f(p) = -\frac{15d}{2(3/5-d)} \alpha_{ij}=-\frac{15}{4}\alpha_{ij}$ for all $p\in SG\setminus V_0$ so that \eqref{eq:Delta-f-formula} holds. This completes the proof of the theorem.
\end{proof}

The following corollary directly follows from the above theorem and the uniqueness of the solution of the Dirichlet problem (see Theorem~2.6.1 in \cite{Str}).
\begin{corollary}
   Let $a_1,a_2,a_3$ and $\eta$ be given real numbers. Then the
 Dirichlet problem on SG
 \begin{equation*}
 \left\{
   \begin{array}{ll}
     u(q_i)=a_i, & i=1,2,3, \\
     \Delta u(x)=\eta, & x \in SG \setminus V_0
   \end{array}
 \right.
 \end{equation*}
has the unique solution $u$, which is the FIF with uniform vertical scaling factor $d=\frac{1}{5}$ and satisfying that $u(q_i)=a_i$, $i=1,2,3$, and
$u(q_{ij}) = \frac{1}{5}(2a_i+2a_j+a_{6-i-j}-\frac{\eta}{3})$ for all distinct $i$ and $j$.
 \end{corollary}

\end{document}